\documentclass[11pt]{amsart}

\usepackage{geometry}                
\usepackage{graphicx}
\usepackage{tikz}
\usepackage{tikz-cd}
\usepackage{amssymb}
\usepackage{csquotes}
\usepackage{epstopdf}
\usepackage{amsmath}
\usepackage{color}
\usepackage{url}
\usepackage{stmaryrd}
\usepackage{mathtools}
\usepackage{array}
\usetikzlibrary{matrix}

\theoremstyle{definition}

\newtheorem*{thm*}{Theorem}

\theoremstyle{plain}

\newtheorem{thm}{Theorem}[section]
\newtheorem{lem}[thm]{Lemma}
\newtheorem{prop}[thm]{Proposition}
\newtheorem{corol}[thm]{Corollary}

\newtheorem{rem}[thm]{Remark}
\numberwithin{equation}{section}
\makeatletter
\newcommand*{\@old@slash}{}\let\@old@slash\slash
\def\slash{\relax\ifmmode\delimiter"502F30E\mathopen{}\else\@old@slash\fi}
\makeatother

\def\backslash{\delimiter"526E30F\mathopen{}}

\DeclareMathOperator{\Pic}{Pic}

\DeclareMathOperator{\Gal}{Gal}

\DeclareMathOperator{\GL}{GL}
\DeclareMathOperator{\SL}{SL}

\newcommand{\N}{\mathbb{N}}
\newcommand{\Q}{\mathbb{Q}}
\newcommand{\Z}{\mathbb{Z}}

\newcommand{\F}{\mathbb{F}}

\DeclareMathOperator{\PGL}{PGL}

\newcommand{\loc}{{\mathcal O}}

\DeclareMathOperator{\Hom}{Hom}

\DeclareMathOperator{\Id}{Id}

\DeclareMathOperator{\Cot}{Cot}

\DeclareMathOperator{\Tate}{Tate}

\DeclareMathOperator{\Tan}{Tan}
\DeclareMathOperator{\Isom}{Isom}

\renewcommand{\b}{\beta}

\newcommand{\Hcal}{{\mathcal H}}

\newcommand{\Ocal}{{\mathcal O}}

\renewcommand{\mod}{\, \operatorname{mod} \,}

\newcommand{\Qb}{\overline{\Q}}
\newcommand{\GalQ}{{\Gal(\Qb / \Q)}}

\newcommand{\mt}{\mapsto}

\title[Residual Galois representations and normaliser of nonsplit Cartan]{Residual Galois representations of elliptic curves with image contained in the normaliser of a non-split Cartan}
\author{Samuel Le Fourn}
\address{Univ. Grenoble Alpes, CNRS, IF, 38000 Grenoble, France}
\email{Samuel.Le-Fourn@univ-grenoble-alpes.fr}
\author{Pedro Lemos}
\thanks{The second named author is funded by the Royal Society Research Fellows Enhancement Award RGF\textbackslash EA\textbackslash 181052.}
\address{Department of Mathematics, University College London, 25 Gordon Street, WC1H 0AY London, United Kingdom}
\email{lemos.pj@gmail.com}
\begin{document}
	\begin{abstract}
		It is known that if $p>37$ is a prime number and $E/\Q$ is an elliptic curve without complex multiplication, then the image of the mod $p$ Galois representation \begin{equation*} \bar{\rho}_{E,p}:\Gal(\overline{\Q}/\Q)\rightarrow \GL(E[p]) \end{equation*}
		of $E$ is either the whole of $\GL(E[p])$, or is \emph{contained} in the normaliser of a non-split Cartan subgroup of $\GL(E[p])$. In this paper, we show that when $p>1.4\times 10^7$, the image of $\bar{\rho}_{E,p}$ is either $\GL(E[p])$, or the \emph{full} normaliser of a non-split Cartan subgroup. We use this to show the following result, partially settling a question of Najman. For $d\geq 1$, let $I(d)$ denote the set of primes $p$ for which there exists an elliptic curve defined over $\Q$ and without complex multiplication admitting a degree $p$ isogeny defined over a number field of degree $\leq d$. We show that, for $d\geq 1.4\times 10^7$, we have
		\begin{equation*}
		I(d)=\{p\text{ prime}:p\leq d-1\}.
		\end{equation*}
	\end{abstract}
	
\maketitle

\section{Introduction}\label{intro}

Let $p$ be a prime, and let $E$ be an elliptic curve defined over $\Q$. Fix an algebraic closure $\overline{\Q}$ of $\Q$, and denote by $E[p]$ be the group of $p$-torsion points of $E(\overline{\Q})$.  This is a 2-dimensional $\F_p$-vector space endowed with an $\F_p$-linear action of the Galois group $G_\Q := \GalQ$.  
 We thus have an associated Galois representation
\begin{equation*}
\bar{\rho}_{E,p}:G_\Q\rightarrow \GL(E[p]).
\end{equation*}

 When $E$ does not have complex multiplication, Serre~\cite{ser_prop} shows that, for $p$ large enough, the image of $\bar{\rho}_{E,p}$ is the whole of $\GL(E[p])$. In the same paper~\cite{ser_prop}, he asks whether it is possible to prove a uniform lower bound exists for his result to hold, i.e. whether there exists a positive constant $B$ such that if $E/\Q$ is an elliptic curve without complex multiplication and $p$ is a prime larger than $B$, then $\bar{\rho}_{E,p}$ is surjective. This problem is commonly known as Serre's uniformity question. The progress made towards finding an answer to it can be summarised in the following result, due to the work of several mathematicians, amongst whom we highlight Bilu, Mazur, Parent, Rebolledo and Serre (the terminology used in the statement of the following theorem will be explained in the next section).

\begin{thm}[\cite{maz_rat, bilpar,bilparreb,ser_que}]\label{soa}
Let $E/\Q$ be an elliptic curve without complex multiplication. Suppose that $p$ is a prime not lying in the set $\{2,3,5,7,11,13,17,37\}$. If the image of $\bar{\rho}_{E,p}$ is not $\GL(E[p])$, then it is \emph{contained} in the normaliser of a non-split Cartan subgroup of $\GL(E[p])$.
\end{thm}

The main result of this paper is the following improvement on Theorem~\ref{soa}.
\begin{thm}\label{main}
Let $E/\Q$ be an elliptic curve without complex multiplication. Let $p$ be a prime number, and suppose that one of the following statements holds:

$(a)$ $p > 1.4 \times 10^7$;

$(b)$  $p \notin \{2,3,5,7,11,13,17,37\}$ and $j(E) \notin \Z$.

\noindent If $\bar{\rho}_{E,p}$ is not surjective, then its image \emph{is} the normaliser of a non-split Cartan subgroup of $\GL(E[p])$.
\end{thm}

We will at times mention Theorem~\ref{main} $(a)$ or Theorem~\ref{main}~$(b)$, by which we mean the result of Theorem~\ref{main} assuming condition $(a)$ or $(b)$, respectively.


The proof of Theorem \ref{main} $(a)$ (section \ref{secRunge}) shows, in fact, that if an elliptic curve $E$ without  complex multiplication is such that $j(E) \in \Z$ and admits a prime $p$ not in the set $\{2,3,5,7,11,13,17,37\}$ such that the image of $\bar{\rho}_{E,p}$ is neither $\GL(E[p])$ nor the normaliser of a nonsplit Cartan subgroup, then $\log |j(E)| \leq \max(12000, 7 \sqrt{p}) \leq 27000$. In particular, there are only finitely many such elliptic curves up to isomorphism. One would could then hope that the remaining cases might be treated algorithmically, but the authors admit they could not find a reasonably efficient way to do so. However, we wish to point out that some work has already been done in this direction. For example, in \cite{BajoletBiluMatschke2018}, the integral points of $X_{\rm{ns}}^+(p)$ are determined for all primes $p \leq 97$. Unfortunately, the algorithms employed there (which are already great improvements over existing techniques) need several CPU years to compute even the single case $p=97$. Solving the remaining cases $p \leq 1.4 \cdot 10^7$ in our case thus appears as a serious technical challenge deserving of its own project.

As an immediate application of Theorem~\ref{main}, we are able to partially settle a question of Najman~\cite{naj}. Let $d\geq 1$ be a positive integer. Najman~\cite{naj} defines $I(d)$ to be the set of primes $p$ for which there exists an elliptic curve $E$ defined over $\Q$ without complex multiplication and an isogeny $\phi:E_{/K}\rightarrow E'$ of degree $p$ defined over a number field $K$ of degree $\leq d$. For instance, a celebrated result of Mazur~\cite{maz_rat} states that
\begin{equation*}
I(1)=\{2,3,5,7,11,13,17,37\}.
\end{equation*}
Najman~\cite{naj} shows that
\begin{equation*}
\begin{aligned}
I(d)\subseteq&I(1)\cup\{p\text{ prime}: p\leq d-1 \text{ when }p\equiv 1\pmod{3}\}\cup \\ &\{p\text{ prime}: p\leq 3d-1 \text{ when }p\equiv 2\pmod{3}\}.
\end{aligned}
\end{equation*}
Assuming that $\bar{\rho}_{E,p}$ is surjective whenever $E/\Q$ does not have complex multiplication and $p \notin I(1)$, he proves that one has
\begin{equation}\label{isoeq}
I(d)=I(1)\cup\{p\text{ prime}: p\leq d-1 \}.
\end{equation}
Theorem~\ref{main} allows us to remove the condition on the surjectivity of $\bar{\rho_{E,p}}$, albeit adding one on the size of $d$.
\begin{thm}
\label{isothm}
For $d\geq 1.4\times 10^7$, we have $I(d)=\{p\text{ prime}: p\leq d-1 \}$.
\end{thm}
The proof of this result is a simple combination of Theorem~\ref{main} and Najman's own arguments. We refer the reader to~\cite{naj} for details.

\subsection{Plan of the proof of Theorem \ref{main}}

From now on, we assume that a basis of $E[p]$ has been chosen and systematically identify $\GL(E[p])$ with $\GL_2(\F_p)$.

The following result of Zywina~\cite[Proposition 1.13]{zyw} will be used to prove Theorem~\ref{main}.

\begin{prop}[Zywina]\label{zywina}
Let $E/\Q$ be an elliptic curve without complex multiplication. Let $p\notin I(1)$ be a prime such that $\bar{\rho}_{E,p}$ is not surjective. 

(1) If $p\equiv 1\pmod{3}$, then $\bar{\rho}_{E,p}(G_{\Q})$ is the normaliser of a non-split Cartan subgroup of $\GL_2(\F_p)$.

(2) If $p\equiv 2\pmod{3}$, then $\bar{\rho}_{E,p}(G_{\Q})$ is either the normaliser of a non-split Cartan subgroup of $\GL_2(\F_p)$, or is conjugate in $\GL_2(\F_p)$ to the group
\begin{equation*}
G(p):=\left\{a^3: a\in C_{\rm{ns}}(p)\right\}\cup\left\{\begin{pmatrix}1 & 0 \\ 0 & -1\end{pmatrix}\cdot a^3: a\in C_{\rm{ns}}(p)\right\},
\end{equation*}
where $C_{\rm{ns}}(p)$ is an explicit choice of non-split Cartan subgroup made in the next section.
\end{prop}
\begin{rem}
For the convenience of the reader, and following a suggestion made by an anonymous referee, we reproduce Zywina's proof of this result in Appendix~\ref{appb}.
\end{rem}
We are then reduced to showing that when $p\equiv 2\pmod{3}$ the image of $\bar{\rho}_{E,p}$ cannot be contained in a conjugate of $G(p)$. When the $j$-invariant of $E/\Q$ is not integral, we will rule out this possibility using Mazur's formal immersion argument (see~\cite{maz_rat}). More precisely, an elliptic curve defined over $\Q$ whose residual Galois representation $\bar{\rho}_{E,p}$ has image contained in $G(p)$ will give rise to a $\Q$-rational point $x$ on a modular curve $X_{G(p)}$. If the $j$-invariant is not in $\Z$, then some prime $\ell$ divides the denominator. We will first point out that $\ell$ cannot be $p$ (this is Proposition~\ref{notp}). It will then follow that there exists a $\Q$-rational point $x$ in the residue class modulo $\lambda$ of a cusp $c$ (here, $\lambda$ is a prime of the residue field of the cusp $c$ dividing $\ell$). We will show the existence of a non-trivial quotient of the jacobian of $X_{G(p)}$ with finite Mordell--Weil group (this is Section~\ref{quotMW}) and use the standard formal immersion arguments due to Mazur to prove that such a point cannot exist (Sections~\ref{formalimmersions} and~\ref{thmb}). This will give us Theorem~\ref{main} $(b)$. 

\begin{rem} The reader will notice that this situation contrasts with that of the modular curve $X_{\rm{ns}}^+(p)$ associated to the normaliser of a non-split Cartan. Indeed, it is well-known that the conjecture of Birch and Swinnerton-Dyer implies the inexistence of a non-trivial quotient of the jacobian of $X_{\rm{ns}}^+(p)$ with finite Mordell--Weil group. This is a major obstacle to the study of the rational points of $X_{\rm{ns}}^+(p)$, and thus to giving a positive answer to Serre's uniformity question.
\end{rem}

In the case where $j(E) \in \Z$,  the assumptions on the mod $p$ Galois representation of $E$ give rise to an integral point on $X_{G(p)}$. We then follow the steps of Bilu and Parent in \cite{bilpar} as follows. First, by applying Runge's method, we obtain an upper bound for $\log |j(E)|$ which is linear in $\sqrt{p}$. An explicit version of Serre's surjectivity theorem obtained by the first author on the basis of isogeny theorems of Gaudron and Rémond \cite{GaudronRemond14} provides a lower bound linear in $p$, which gives rise to a contradiction  for $p \geq 1.4 \times 10^7$.




\section{Cusps of modular curves}\label{modcurv}
We give a brief review of some basic facts about cusps of modular curves and set down some notation that will be used later in the paper. The reader should be warned that our definition of \emph{cusps at infinity} differs slightly from the standard one.

Let $p$ be an odd prime, and let $X(p)$ be the (compactification of the) classical modular curve which classifies pairs $(E,(P,Q))$, where $E$ is an elliptic curve and the pair $(P,Q)$ is an $\F_p$-basis of $E[p]$. This is a smooth projective curve over $\Q$ whose base change to $\Q(\zeta_p)$ has $p-1$ connected components. Given a subgroup $H$ of $\GL_2(\F_p)$, we will denote the modular curve $H\backslash X(p)$ by $X_H$.

Define
\begin{equation}
\label{eqdefMp}
M_p := \left( (\Z/p\Z)^2-\{(0,0)\} \right)/\pm 1.
\end{equation}

If we regard the elements of $M_p$ as column vectors, we have a natural left action of $\GL_2(\F_p)$ on $M_p$. We can therefore define an action of $\GL_2(\F_p)$ on $M_p\times\F_p^{\times}$ by letting $\GL_2(\F_p)$ act on $\F_p^{\times}$ via multiplication by the determinant. 
\begin{lem}\label{lemcusps}
There is a bijection between the cusps of $X(p)$ and the set $M_p\times\F_p^{\times}$ which is equivariant for the action of $\GL_2(\F_p)$. Moreover, if $\sigma\in G_{\Q}$ and $c$ is a cusp of $X(p)$ corresponding to the pair $\left(\begin{pmatrix} a \\ b\end{pmatrix}, d\right)$, then ${}^{\sigma} c$ corresponds to 
\[
\sigma \cdot \left( \begin{pmatrix} a \\ b \end{pmatrix}, d \right) :=\left( \chi_p (\sigma)^{-1} \begin{pmatrix} a \\ b \end{pmatrix}, \chi_p(\sigma)^{-1} d \right),
\]
where $\chi_p$ is the cyclotomic character.
\end{lem}
\begin{proof}
Following \cite[VI.5]{delrap}, we have a canonical Galois equivariant bijection between the cusps of $X(p)$ and the set
\begin{equation*}
    \Isom(\mu_p\times\Z/p\Z, (\Z/p\Z)^2)/\pm U,
\end{equation*}
where $U$ is the group of matrices
\begin{equation*}
    \begin{pmatrix} 1 & u\\ 0 & 1\end{pmatrix},\qquad u\in\Hom(\Z/p\Z,\mu_p),
\end{equation*}
and the (left) action of $\GalQ$ is induced by its natural action on $\mu_p$ and trivial one on $\Z/p\Z$. 
Furthermore, the action of $\GL_2(\F_p)$ corresponds to composition (in other words, to left matrix multiplication).

Given a class $\gamma$ in $\Isom(\mu_p\times\Z/p\Z, (\Z/p\Z)^2)/\pm U$ represented by 
\begin{equation*}
    (\zeta_p,0)\mapsto (a,b),\quad (1,1)\mapsto (c,d),
\end{equation*}
we associate to it the element 
\begin{equation*}
    \left(\begin{pmatrix} a\\ b\end{pmatrix}, \det\gamma\right)\in M_p\times \F_p^{\times},
\end{equation*}
where the determinant of $\gamma$ is defined to be $ad-bc$. It is easy to see that this function is well defined. It is also clear that this function commutes with the actions of $\GL_2(\F_p)$ and of the Galois group, so we just need to check that it is a bijection. But this is clearly surjective, as, given a pair $(a,b)\in (\Z/p\Z)^2-\{(0,0)\}$, it is always possible to find a pair $(c,d)$ such that $ad-bc$ is equal to a given element of $\F_p^{\times}$. As the two sets have the same number of elements, the result follows.
\end{proof}

\begin{corol}\label{corolcusps}
If $H$ is a subgroup of $\GL_2(\F_p)$, then there is a bijection between the set of cusps of $X_H$ and the set $H\backslash(M_p\times\F_p^{\times})$. Moreover, if $\det H=\F_p^{\times}$, this bijection induces a bijection between the set of cusps of $X_H$ and $(H\cap\SL_2(\F_p))\backslash M_p$.
\end{corol}
\begin{proof}
The first assertion follows immediately from Lemma~\ref{lemcusps} and the definition of $X_H$. In order to prove the second one, start by observing that, given a class in $H\backslash(M_p\times\F_p^{\times})$, there is always a representative of this class whose second entry is $1$ (this is due to the assumption that $\det H=\F_p^{\times}$). Therefore, the map $(H\cap\SL_2(\F_p))\backslash M_p\rightarrow H\backslash(M_p\times\F_p^{\times})$ given by
\begin{equation*}
    \begin{pmatrix} a \\ b\end{pmatrix}\mapsto\left(\begin{pmatrix} a \\b\end{pmatrix}, 1\right)
\end{equation*}
is well-defined and bijective.
\end{proof}

\begin{corol}\label{galcusps2}
Let $H$ be a subgroup of $\GL_2(\F_p)$ such that $\det H=\F_p^{\times}$. Under the identification of Corollary~\ref{corolcusps}, the Galois orbit of a cusp of $X_H$ represented by an element  $\begin{pmatrix} a \\ b\end{pmatrix}$ of $M_p$ is the set of cusps of $X_{H}$ represented by the elements in the set
\begin{equation*}
    \left\{\gamma \cdot \begin{pmatrix}  a\\  b\end{pmatrix}\in M_p: \, \gamma\in H\right\}.
\end{equation*}
In particular, we obtain a one-to-one correspondence between the Galois orbits of cusps of $X_H$ and the set $H\backslash M_p$.
\end{corol} 
\begin{proof}
For each $\lambda\in \F_p^{\times}$, choose $\gamma_{\lambda}\in H$ such that $\det\gamma_{\lambda}=\lambda$. The first observation we want to make is that we have the following equality of sets: 
\begin{equation}\label{observation}
    H=\{ h\gamma_{\lambda}:\lambda\in\F_p^{\times}, h\in H\cap\SL_2(\F_p)\}.
\end{equation}
Indeed, if $g\in H$, and if we set $d:=\det g$, we have $g\gamma_d^{-1}\in H\cap \SL_2(\F_p)$. Thus, $g$ is of the form $h\gamma_d$ for some $h\in H\cap\SL_2(\F_p)$. The other inclusion is obvious.

According to Lemma~\ref{lemcusps}, the Galois orbit of a cusp represented by $\left(\begin{pmatrix} a \\ b \end{pmatrix}, 1\right)$ is the set of cusps represented by the elements of the set
\begin{equation*}
    \Sigma:=\left\{\left(\begin{pmatrix} \lambda a\\ \lambda b\end{pmatrix},\lambda \right)\in  M_p\times\F_p^{\times}:\lambda\in\F_p^{\times}\right\}.
\end{equation*}
Of course, the set of cusps of $X_H$ represented by $\Sigma$ and the set of those represented by \[\gamma_{\lambda}^{-1}\cdot\Sigma=\left\{\left(\lambda\gamma_{\lambda}^{-1}\begin{pmatrix}a\\b\end{pmatrix},1\right)\in M_p\times\F_p^{\times}:\lambda\in \F_p^{\times}\right\}\] is the same. As $\det H=\F_p^{\times}$, we know from Corollary~\ref{corolcusps} that we can also identify the set of cusps of $X_H$ with $(H\cap\SL_2(\F_p))\backslash M_p$. Therefore, the Galois orbit of our cusp is the set of cusps of $X_H$ represented by the elements of the set
\begin{equation}\label{galorbit}
    \left\{\lambda\gamma_{\lambda}^{-1} \cdot \begin{pmatrix} a \\ b\end{pmatrix}\in  M_p:\lambda\in \F_p^{\times}\right\}.
\end{equation}
As $\det\lambda\gamma_{\lambda}^{-1}=\lambda$, we see that $\{\lambda\gamma_{\lambda}^{-1}:\lambda\in \F_p^{\times}\}$ runs through a set of representatives of $(H\cap\SL_2(\F_p))\backslash H$. Equality (\ref{observation}) and the fact that the cusps of $X_H$ represented by a set of elements of $M_p$ does not change under the action of an element of $H\cap\SL_2(\F_p)$,  yield that the Galois orbit of our cusp is represented by the set
\begin{equation*}
    \left\{\gamma \cdot \begin{pmatrix} a \\ b\end{pmatrix} \in  M_p : \gamma\in H\right\},
\end{equation*}
as we wanted.
\end{proof}

We define the \emph{cusps at infinity} of a modular curve $X_H$ to be those cusps represented by elements of $M_p\times\F_p^{\times}$ of the form
\[ \left(\begin{pmatrix} a\\ 0\end{pmatrix},a \right),\quad a\in\F_p^{\times}\]under the identification of Corollary~\ref{corolcusps}. Note that the set of cusps at infinity of $X_H$ forms a full Galois orbit.



Before finishing this section, we wish to mention some of the modular curves that we will use throughout this article. We start by considering the case where $H$ is the upper triangular subgroup of $\GL_2(\F_p)$. In this case, the curve $X_H$ is usually denoted by $X_0(p)$. 
This modular curve has two distinct cusps: one cusp at infinity and one not at infinity, as one can easily check using the identification of Corollary~\ref{corolcusps}. Both cusps are defined over~$\Q$. The cusp at infinity will be denoted by $\infty$, while the other one will be denoted by~$0$.

Let $C_{\rm{sp}}(p)$ be the split Cartan subgroup of $\GL_2(\F_p)$ consisting of diagonal matrices, i.e.,
\begin{equation*}
C_{\rm{sp}}(p):=\left\{\begin{pmatrix} a & 0\\ 0 & b\end{pmatrix}\in\GL_2(\F_p): a,b\in\F_p^{\times}\right\}.
\end{equation*}
When $H = C_{\rm{sp}}(p)$, we will denote $X_H$ by $X_{\rm{sp}}(p)$. The normaliser of $C_{\rm{sp}}(p)$ will be denoted by $N_{\rm{sp}}(p)$. This is the subgroup of $\GL_2(\F_p)$ consisting of diagonal and anti-diagonal matrices. When $H=N_{\rm{sp}}(p)$, the curve $X_H$ will be denoted by $X_{\rm{sp}}^+(p)$. We have a canonical morphism $X_{\rm{sp}}(p)\rightarrow X_{\rm{sp}}^+(p)$ of degree $2$, which is unramified at the cusps. The curve $X_{\rm{sp}}^+(p)$ has $(p+1)/2$ cusps, of which exactly one is at infinity. The cusp at infinity is defined over $\Q$, while the others are defined over $\Q(\zeta_p)^+:=\Q(\zeta_p+\zeta_p^{-1})$. The pre-image of the cusp at infinity of $X_{\rm{sp}}^+(p)$ in $X_{\rm{sp}}(p)$ consists of two cusps, both defined over $\Q$, of which one is at infinity and the other is not. The remaining $p-1$ cusps of $X_{\rm{sp}}(p)$ are defined over $\Q(\zeta_p)$.

Finally, we make the following choice of non-split Cartan subgroup of $\GL_2(\F_p)$. Fix a generator $\epsilon_p$ of the cyclic group $\F_p^{\times}$. Define
\begin{equation*}
C_{\rm{ns}}(p):=\left\{\begin{pmatrix} a & \epsilon_p b\\ b & a\end{pmatrix}\in\GL_2(\F_p): a,b\in\F_p^{\times}\right\}.
\end{equation*}
Its normaliser will be denoted by $N_{\rm{ns}}(p)$. Explicitly, this is given by

\begin{equation*}
N_{\rm{ns}}(p)=C_{\rm{ns}}(p)\cup \begin{pmatrix} 1 & 0\\ 0 & -1\end{pmatrix} C_{\rm{ns}}(p).
\end{equation*}
We have a canonical finite morphism $X_{\rm{sp}}(p)\rightarrow X_{\rm{sp}}^+(p)$ of degree $2$. The modular curve $X_{\rm{ns}}^+(p)$ has $(p-1)/2$ cusps, all of them at infinity. Their field of definition is~$\Q(\zeta_p)^+$. The modular curve $X_{\rm{ns}}(p)$, on the other hand, has $p+1$ distinct cusps. Like in the case of~$X_{\rm{ns}}^+(p)$, they are all at infinity, but their field of definition is now $\Q(\zeta_p)$.

\section{Quotients of modular jacobians with finite Mordell--Weil group}\label{quotMW}

Let $p$ be an odd prime, and let $H(p)$ denote the group $N_{\rm{ns}}(p)\cap N_{\rm{sp}}(p)$. Let $H$ be a subgroup of $N_{\rm{ns}}(p)$ of index $d\geq 2$ containing $H(p)$. We shall write $J_H$ for the jacobian of the modular curve $X_H$. The main result of this section is the following proposition.
\begin{prop}\label{MW}
Suppose that $p=11$ or $p\geq 17$. Then the jacobian $J_H$ of $X_H$ admits a non-trivial optimal quotient $A$ such that

(1) $A(\Q)$ is finite;

(2) the kernel of the canonical projection $J_{H}\rightarrow A$ is stable under the action of the Hecke operators $T_{\ell}$, where $\ell\neq p$ is a prime.
\end{prop}

Let $\pi_{\rm{ns}}:X_{H(p)}\rightarrow X_{\rm{ns}}^+(p)$ denote the canonical projection. As $H$ is contained in $N_{\rm{ns}}(p)$ and contains $H(p)$ by assumption, the morphism $\pi_{\rm{ns}}$ factors through $X_H$:

\begin{equation}\label{factor}
\begin{tikzcd}
                 & X_{H(p)} \arrow[d, "\pi_{\rm{ns}}"] \arrow[ld, "\pi_H"'] \\
X_{H} \arrow[r, "\pi'_H"'] & X_{\rm{ns}}^+(p)                                 
\end{tikzcd}.
\end{equation}
As the index of $H$ in $N_{\rm{ns}}(p)$ is $d$, we have $\deg \pi'_H = d$ and  $\deg(\pi_H)=(p+1)/(2d)$. Moreover, the morphisms $\pi_H$ and $\pi_H'$ are unramified at the cusps because $\pi_{\rm{ns}}$ is.
\begin{lem}\label{notinf}
The modular curve $X_H$ has cusps not at infinity.
\end{lem}
\begin{proof}
 In the language of Corollary~\ref{corolcusps}, the set of cusps of $X_{H(p)}$ is identified with the set \[H(p)\backslash M_p\times \F_p^{\times},\] and the cusps at infinity are, by definition, those represented by an element of the form
 \[\left(\begin{pmatrix} a\\ 0\end{pmatrix},a\right),\quad a\in\F_p^{\times}.\] It follows that $X_{H(p)}$ has $(p-1)/2$ cusps at infinity. Therefore, $X_H$ has at most $(p-1)/2$ cusps at infinity. As $\pi_H'$ is a morphism of degree $d$ unramified at the cusps, the number of cusps of $X_H$ is $d(p-1)/2$ (recall that $X_{\rm{ns}}^+(p)$ has $(p-1)/2$ cusps). Since $d\geq 2$ by assumption, this number is strictly larger than $(p-1)/2$. As a consequence, there exists a cusp of $X_H$ which is not at infinity.
\end{proof}
\begin{rem}
 Note that the proof of Lemma~\ref{notinf} relies crucially on the fact that $d>1$. If $d=1$, then $X_H$ is the modular curve $X_{\rm{ns}}^+(p)$ and, indeed, all of its cusps are at infinity.
\end{rem}
 Let $\pi_{\rm{sp}}:X_{H(p)}\rightarrow X_{\rm{sp}}^+(p)$ be the canonical projection. By pulling back by $\pi_H$ and pushing forward via $\pi_{\rm{sp}}$, we obtain a morphism
\[\phi:=\pi_{\rm{sp},*}\circ \pi_H^{*}:X_H\rightarrow \Pic(X_{\rm{sp}}^+(p)). \]
\begin{lem}\label{newlem}
If $c\in X_{H(p)}(\Q(\zeta_p))$ is not a cusp at infinity, then $\pi_{\rm{sp}}(c)$ is not at infinity.
\end{lem}
\begin{proof}
 Let $\left(\begin{pmatrix} a\\b \end{pmatrix},d\right)$ be an element of $M_p\times\F_p^{\times}$ representing $c$. Suppose, for contradiction, that $\pi_{\rm{sp}}(c)$ is at infinity. Then there exists $\gamma\in N_{\rm{sp}}(p)$ and $\alpha\in\F_p^{\times}$ such that 
 \[ \gamma\cdot\begin{pmatrix} a\\ b\end{pmatrix}=\begin{pmatrix} \alpha\\ 0\end{pmatrix}.\] It follows that $a=0$ or $b=0$. However, the cusps of $X_{H(p)}$ represented by elements of $M_p\times \F_p^{\times}$ of the form \[ \left(\begin{pmatrix} a\\ 0\end{pmatrix},d \right)\quad\text{or}\quad\left(\begin{pmatrix} 0\\ b\end{pmatrix},d\right)\]are easily seen to be cusps at infinity, which is a contradiction.
\end{proof}

We also recall a well-known morphism
\[\eta: X_{\rm{sp}}^+(p)\rightarrow J_0(p)/(1+w_p)J_0(p)\] (where $w_p$ is the Atkin--Lehner involution of $J_0(p)$) that has been studied, for instance, by Mazur in~\cite{maz_rat} and by Momose in~\cite{momose}. This morphism is defined as follows. Start by considering the two degeneration maps 
\[d_1: X_{\rm{sp}}(p)\rightarrow X_0(p)\quad\text{and}\quad d_p:X_{\rm{sp}}(p)\rightarrow X_0(p)\] whose moduli interpretations are
\[ d_1:(E,(A,B))\mapsto (E,A)\quad\text{and}\quad d_p:(E,(A,B))\mapsto (E/B, E[p]/B).\] We define an auxiliary morphism
\[\eta':X_{\rm{sp}}(p)\rightarrow J_0(p)\] by requiring that  a point $x\in X_{\rm{sp}}(p)(\overline{\Q})$ be mapped to the class of the divisor $d_1(x)-d_p(x)$. Letting $\omega_p$ denote the involution of $X_{\rm{sp}}(p)$ whose moduli interpretation is \[(E,(A,B))\mapsto (E,(B,A)),\] (where $A$ and $B$ are distinct subgroups of $E(\overline{\Q})$ of order $p$) it is easy to check that the following relations hold:
\begin{equation}\label{idents} w_p\circ d_p=d_1\circ \omega_p\quad\text{and} \quad w_p\circ d_1=d_p\circ \omega_p.\end{equation} It follows that \[\eta'\circ\omega_p=-w_p\circ \eta'.\]Therefore, the morphism
\[X_{\rm{sp}}(p)\xrightarrow{\eta'}J_0(p)\rightarrow J_0(p)/(1+w_p)J_0(p)\] factors through $X_{\rm{sp}}^+(p)$. The morphism $X_{\rm{sp}}^+(p)\rightarrow J_0(p)/(1+w_p)J_0(p)$ through which it factors is the morphism~$\eta$ aforementioned.

Despite the following lemma being a well-known result, the authors were not able to find a reference offering a concise proof. Due to this, a proof of this lemma can be found in Appendix~\ref{appa}. 
\begin{lem}\label{imgeta}
Let $c\in X_{\rm{sp}}^+(p)(\Q(\zeta_p))$ be a cusp. We have
\[ \eta(c)=\begin{cases} 0 &\text{if $c$ is at infinity}  \\
\mathrm{cl}(0-\infty) & \text{otherwise}.
\end{cases}\]
\end{lem}

By abuse of notation, we shall denote by $\eta$ the map $\Pic(X_{\rm{sp}}^+(p))\rightarrow J_0(p)/(1+w_p)J_0(p)$ obtained from $\eta$ using the universal property of jacobians. Define $\nu$ to be the composition
\[\eta\circ\phi:X_H\rightarrow J_0(p)/(1+w_p)J_0(p),\] which is clearly a morphism defined over $\Q$. 
\begin{lem}\label{key}
Let $c\in X_H(\Q(\zeta_p))$ be a cusp. If $c$ is a cusp at infinity, we have
\[ \nu(c)=\left(\frac{p+1}{2d}-1\right)\mathrm{cl}(0-\infty).\]If, on the other hand, $c$ is not at infinity,  we have
\[\nu(c)=\left( \frac{p+1}{2d} \right) \mathrm{cl}(0-\infty).\]
\end{lem}
\begin{proof}
 If $c$ is not at infinity, the the pull-back of $c$ by $\pi_H$ is a sum of $(p+1)/2d$ cusps of $X_{H(p)}$ not at infinity. Then, using Lemma~\ref{newlem}, we conclude that $\phi(c)$ is a sum of $(p+1)/2d$ cusps of $X_{\rm{sp}}^+(p)$ not at infinity. The image of this divisor under $\eta$ is then 
 \[\left( \frac{p+1}{2d} \right)\mathrm{cl}(0-\infty)\]
 by Lemma~\ref{imgeta}.
 
 If $c$ is a cusp at infinity, then, using the language of Corollary~\ref{corolcusps}, it is represented by an element of $M_p\times \F_p^{\times}$ of the form
 \[\left(\begin{pmatrix} a\\ 0\end{pmatrix},a \right),\quad a\in \F_p^{\times}.\] The pullback of $c$ by $\pi_H$ corresponds then to pulling back this element by the map
 \[H(p)\backslash M_p\times\F_p^{\times}\rightarrow H\backslash M_p\times\F_p^{\times}.\] There is only one cusp at infinity of $X_{H(p)}$ in the pre-image of $c$. Indeed, the cusps in the pre-image of $c$ are all of the form $\gamma c$ for some $\gamma\in H\subseteq N_{\rm{ns}}(p)$. But the only elements of $N_{\rm{ns}}(p)$ fixing the line generated by $\begin{pmatrix} a \\ 0\end{pmatrix}$ are in $H(p)$, which proves the claim. It then follows from Lemma~\ref{newlem} that $\phi(c)$ is the sum of one cusp at infinity with $(p+1)/2d-1$ cusps not at infinity. Using Lemma~\ref{imgeta}, we conclude that
 \[\nu(c)=\left(\frac{p+1}{2d}-1\right)\mathrm{cl}(0-\infty),\]
 as we wanted.
\end{proof}

 Let $\tilde{J}_p$ be the Eisenstein quotient of $J_0(p)$ (see~\cite[II.10 Definitions 10.4]{maz_eis}), and let $\vartheta$ be the composition
 \[ J_H\rightarrow J_0(p)/(1+w_p)J_0(p)\xrightarrow{\mathrm{pr}} \tilde{J}_p,\] where the first map is the one induced by $\nu$, and $\mathrm{pr}$ is the natural projection.
\begin{proof}[Proof of Proposition~\ref{MW}]
 Denote the image of $\vartheta$ in $\tilde{J}_p$ by $B$. This is an abelian variety defined over $\Q$, and we need to show that it is not trivial. Let $c$ be a cusp of $X_H$ at infinity, and let $c'$ be one not at infinity (the existence of such a cusp is guaranteed by Lemma~\ref{notinf}). Then, by Lemma~\ref{key}, we have
 \[ \vartheta(\mathrm{cl}(c'-c))=\mathrm{pr}(\mathrm{cl}(0-\infty)).\]
 A theorem of Mazur~\cite[III.1 Corollary 1.4]{maz_eis} now yields that the order of $\mathrm{pr}(\mathrm{cl}(0-\infty))$ is $(p-1)/\gcd(p-1,12)$, which is not $1$ because $p=11$ or $p\geq 17$. This shows that $B$ is not trivial. Also,  $B(\Q)$ is finite because $\tilde{J}_p(\Q)$ is. 
 
 Let $K$ denote the kernel of $\vartheta$, and let $K^0$ denote the connected component of the identity. We define $A$ to be $J_H/K^0$. As $A$ is $\Q$-isogenous to $B$, it follows that it is not trivial and that $A(\Q)$ is finite. This proves statement $(1)$.
 
 Let $\ell$ be a prime different from $p$. One easily checks that the morphism \[J_H\rightarrow J_0(p)/(1+w_p)J_0(p)\] commutes with the action of $T_{\ell}$. By the work of Mazur~\cite{maz_rat}, we already know that the Hecke operator $T_{\ell}$ preserves the kernel of the projection $J_0(p)/(1+w_p)J_0(p)\rightarrow \tilde{J}_p$, from where it follows that it  also preserves $K$. As $T_{\ell}$ is an endomorphism of abelian varieties (and is, in particular, continuous), it maps $K^0$ to itself. This finishes the proof of the proposition.
\end{proof}

\section{Formal immersions}\label{formalimmersions}
For a cusp $c$ of $X_H$, let 
\[\iota_c: X_{H/\Q(\zeta_p)}\rightarrow J_{H/\Q(\zeta_p)}\]be the Abel--Jacobi map centred at $c$. Let $A$ be a non-trivial quotient of $J_H$ satisfying the conditions of Proposition~\ref{MW}. Define $f_c:X_{H/\Q(\zeta_p)}\rightarrow A_{/\Q(\zeta_p)}$ to be the composition \[X_{H/\Q(\zeta_p)}\xrightarrow{\iota_c} J_{H/\Q(\zeta_p)}\rightarrow A_{/\Q(\zeta_p)},\] where the second map is the canonical projection. 

Let $R$ be the ring $\Z[\zeta_p,1/p]$. Let $X_{H/R}$ be the minimal regular model of $X_{H}$ over $R$, and let $A_{/R}$ be the N\'eron model of $A$ over $R$. The N\'eron mapping property allows us to extend the morphism $f_c$ to a morphism \[f_{c/R}:X_{H/R}\rightarrow A_{/R}\] over $R$. If $R'$ is an $R$-algebra, we will write $X_{H/R'}$, $A_{/R'}$ and $f_{c/R'}$ for the base change of $X_{H/R}$, $A_{/R}$ and $f_{c/R}$ to $R'$.

Recall that if $X$ and $Y$ are two noetherian schemes and $\gamma:X\rightarrow Y$ is a morphism, we say that $\gamma$ is a \emph{formal immersion} at the point $x\in X$ if the induced homomorphism \begin{equation*}\hat{\gamma}^{\#}_x:\hat{\loc}_{Y,f(x)}\rightarrow\hat{\loc}_{X,x}\end{equation*} of completed local rings is surjective. 

Given a prime ideal $\lambda$ of $R$, we will write $R_{\lambda}$ for the $\lambda$-adic completion of $R$ at $\lambda$.
\begin{prop}\label{formim}
Let $\lambda$ be a maximal ideal of $R$ whose characteristic is $\neq 2$ (it is also different from $p$ because $p$ is a unit in $R$). The morphism $f_{c/R_{\lambda}}$ is a formal immersion at~$\tilde{c}$, where $\tilde{c}$ stands for the reduction of $c$ modulo~$\lambda$.
\end{prop}

As $\tilde{c}$ and $f(\tilde{c})=0$ are both defined over $k(\lambda)$, the residue field of $R_{\lambda}$, proving this proposition is equivalent to showing that the induced $k(\lambda)$-linear map of cotangent spaces of the special fibres
\begin{equation}\label{surj}
f^*_{c/k(\lambda)}:\Cot(A_{/k(\lambda)})\rightarrow \Cot_{\tilde{c}}(X_{H/k(\lambda)})
\end{equation}
is surjective. As $X_{H/k(\lambda)}$ is $1$-dimensional, it is enough to show that $f^*_{c/k(\lambda)}$ is not trivial. To prove this we will make use of a result due to Mazur~\cite[Lemma~2.1]{maz_rat}, that we now recall.

\subsection{Mazur's lemma} The content of this subsection is completely contained in a more general form in Mazur's paper~\cite{maz_rat}.  Given a cusp $c\in X_{G(p)}(\Q(\zeta_p))$, we will denote by \[\iota_{c/R}:X_{H/R}\rightarrow J_{H/R}\] the extension of the Abel--Jacobi map $\iota_c$ to $R$. We obtain a homomorphism
\begin{equation}\label{iotastar}
\iota_{c/R}^*:\Cot(J_{H/R})\rightarrow\Cot_c(X_{H/R})
\end{equation}
of free $R$-modules. 

A theorem of Raynaud asserts that the Picard variety $\Pic^0_{X_{H}/R}$ of $X_{H/R}$ is canonically identified with $J_{H/R}$. This identification induces an isomorphism between the respective tangent spaces:
\begin{equation*}
i:{\rm H}^1(X_{H/R},\loc_{X_{H}/R})\rightarrow \Tan(J_{H/R}).
\end{equation*}
Of course, $\Cot(J_{H/R})$ is naturally the $R$-dual of $\Tan(J_{H/R})$, while Grothendieck--Serre duality establishes an $R$-duality between ${\rm H}^0(X_{H/R},\Omega^1_{X_H})$ and ${\rm H}^1(X_{H/R},\loc_{X_{H}/R})$. We thus obtain an isomorphism
\begin{equation*}
\Theta:\Cot(J_{H/R})\rightarrow{\rm H}^0(X_{H/R},\Omega^1_{X_{H}/R}).
\end{equation*}
The natural homomorphism $v:{\rm H}^0(X_{H/R},\Omega^1_{X_{H}/R})\rightarrow \Cot_{c}(X_{H/R})$ gives then rise to a homomorphism $v\circ\Theta$ from $\Cot(J_{H/R})$ to $\Cot_{c}(X_{H/R})$. The following lemma, due to Mazur, says that this homomorphism is, up to a sign, $\iota_{c/R}^*$.

\begin{lem}[Mazur~\cite{maz_rat}]\label{mazur}
We have $\iota_{c/R}^*=\pm v\circ\Theta$.
\end{lem}
The reason why the homomorphism $v\circ \Theta$ is so useful is because we can explicitly write $v$ in terms of the $q$-expansion at $c$ of global differential forms of $X_{H}$: the point is that we can identify $\Cot_c(X_{H/R})$ with $R$ in such a way that the diagram 
\begin{equation*}
\begin{tikzcd}
{{\rm H}^0(X_{H/R},\Omega^1_{X_{H}})} \arrow[d, "v"'] \arrow[r, "q\text{-exp}"] & {\Z[[q^{1/p}]]\otimes_{\Z}R} \arrow[d, "\sum a_iq^{i/p}\mapsto a_1"] \\
\Cot_c(X_{H/R}) \arrow[r, "\cong"]                                 & R                                                         
\end{tikzcd}
\end{equation*}
commutes.
\subsection{Formal immersion at the cusps} The last results we need before we are ready to prove Proposition~\ref{formim} are the following two lemmas.
\begin{lem}\label{conj}
Let $\omega$ be an element of ${\rm H}^0(X_{H/\Z[1/p]},\Omega^1_{X_{H}/\Z[1/p]})$. Let $c$ be a cusp of $X_{H}$ and let 
\begin{equation*}
\sum_{n=1}^{\infty} a_n(c,\omega)q^{n/p}\in \Z[[q^{1/p}]]\otimes_{\Z} R
\end{equation*}
be the $q$-expansion of $\omega$ at $c$. If $\sigma\in G_{\Q}$, then $a_n({}^{\sigma}c,\omega)={}^{\sigma}a_n(c,\omega)$.
\end{lem}
\begin{proof}
The result follows from the analogous assertion for $X(p)$, so we show that the lemma holds in this case. Indeed, recall (see~\cite[1.2]{katz}) that the $q$-expansion of a modular form $f$ at a cusp of $X(p)$ is the evaluation of $f$ at the triple $(\Tate(q), \omega_{\rm{can}}, \alpha_p)$, where $\Tate(q)$ is the Tate curve over $\Z((q^{1/p}))\otimes R$, $\omega_{\rm{can}}$ is its canonical differential, and $\alpha_p$ is a $p$-level structure. The result follows from the fact that the formation of $f$ commutes with arbitrary base change (we may take this base change to be the conjugation by an element of the absolute Galois group of $\Q$).
%
\end{proof}
\begin{lem}\label{hecke1}
Let $r$ be a prime number different from $p$, and let $\omega$ be as in the statement of Lemma~\ref{conj}.  Let $T_r$ be the $r$-th Hecke operator. There exists $\sigma\in G_{\Q}$ such that \begin{equation*}
a_1(c,T_r\omega)=a_r({}^{\sigma}c,\omega).
\end{equation*}
\end{lem}
\begin{proof}
By the description of the action of Hecke operators on modular forms of level $p$ given in~\cite[1.11]{katz}, we know that $a_1(c, T_r\omega)=a_r(c', \omega)$ for some cusp $c'$ of $X_{G(p)}$ (be aware that the $\omega$ we are using here is \emph{not} the $\omega$ used in~\cite[1.11]{katz}). In order to actually show that $c'$ must be a conjugate of $c$, let us start by working on $X(p)$. So, let $\omega$ be a modular form of $X(p)$ over $\Z[1/p]$, and let $c$ be a cusp of $X(p)$. The description in~\cite{katz} yields that if $c$ is a cusp of $X(p)$ correponding to the $p$-level structure on a N\'eron $p$-gon given by
\begin{equation*}
    (\zeta_p,0)\mapsto (\alpha,\beta),\quad (1,1)\mapsto (\gamma,\delta),
\end{equation*}
then $a_1(c, T_r\omega)=a_r(c'\omega)$, where $c'$ is the cusp of $X(p)$ correponding to the $p$-level structure
\begin{equation*}
    (\zeta_p,0)\mapsto (\alpha,\beta),\quad (1,1)\mapsto (\gamma',\delta'),
\end{equation*}
where $r\gamma'=\gamma$ and $r\delta'=\delta$ (in $\Z/p\Z$). Thus, using the notation of Lemma~\ref{lemcusps},  if $c$ is a cusp of $X_{H}$ represented by $\left(\begin{pmatrix} \alpha\\ \beta\end{pmatrix},1\right)$ (they all are represented by an element of this form because $\det H=\F_p^{\times}$), then $c'$ will be represented by $\left(\begin{pmatrix}\alpha \\ \beta\end{pmatrix}, r^{-1}\right)$.  Identifying the set $(H\cap\SL_2(\F_p))\backslash M_p$ with the set $H\backslash M_p\times \F_p^{\times}$ as was done in the proof of Corollary~\ref{corolcusps}, the cusp $c$ is represented by $\begin{pmatrix}\alpha \\ \beta\end{pmatrix}$, and $c'$ is represented by $\gamma_r\begin{pmatrix}\alpha \\ \beta\end{pmatrix}$, where $\gamma_r$ is an element of $H$ of determinant~$r$. Corollary~\ref{galcusps2} now yields that $c'$ is in the Galois orbit of $c$.
\end{proof}
\begin{proof}[Proof of Proposition~\ref{formim}]
The proof is standard. As $A_{/k(\lambda)}$ is not trivial, there exists a non-zero $\omega\in\Cot(A_{/k(\lambda)})$. By specialisation results due to Raynaud (as stated, for example, in~\cite[Corollary~1.1]{maz_rat}), the requirement that the characteristic of $\lambda$ is not~$2$ allows us to regard $\Cot(A_{/k(\lambda)})$ as a $k(\lambda)$-linear subspace of $\Cot(J_{H/k(\lambda)})$.

Let
\begin{equation*}
\sum_{n=1}^{\infty} a_n(c,\omega)q^{n/p}\in k(\lambda)[[q^{1/p}]]
\end{equation*}
be the $q$-expansion of $\omega$ at $c_{/ k(\lambda)}$.
Lemma~\ref{mazur} asserts that $f^*_{c/k(\lambda)}(\omega)=\pm a_1(c,\omega)$. If $a_1(c,\omega)\neq 0$, then $f^*_{c/k(\lambda)}$ is not trivial, and we are done. Suppose now that $a_1(c,\omega)=0$. If $r$ is a prime different from $p$, we know that $A$ is stable under the action of $T_r$. Thus, $T_r\omega\in\Cot(A_{/k(\lambda)})$. By Lemma~\ref{hecke1}, there exists $\sigma\in G_{\Q}$ such that $a_1(c, T_r\omega)=a_r({}^{\sigma} c, \omega)$. By Lemma~\ref{conj}, we then have
\begin{equation*}
    a_1(c, T_r\omega)={}^{\sigma}a_r(c, \omega).
\end{equation*}
In particular, $a_1(c, T_{r}\omega)= 0$ if and only if $a_r(c,\omega)=0$. If there exists a prime $r\neq p$ such that $a_r(c,\omega)\neq 0$, we see that $f^*_{c/k(\lambda)}(T_r\omega)\neq 0$ and we are done. We are going to show that such a prime always exists if $a_1(c,\omega)=0$.

Suppose, for the sake of contradiction, that such a prime does not exist. It then follows that $a_n(c,\omega)=0$ for every integer $n\geq 1$ such that $p\nmid n$. Using the $q$-expansion principle, we therefore conclude that $\omega$ is fixed by a conjugate of the group 
\begin{equation*}
U(p):=\left\{\begin{pmatrix} 1 & a \\ 0 & 1\end{pmatrix}\in\GL_2(\F_p):a\in \F_p\right\}\subseteq\GL_2(\F_p).
\end{equation*}
As $\omega$ is also fixed by the action of $H$, it is fixed under the action of the subgroup of $\GL_2(\F_p)$ generated by $H(p)$ and the conjugate of $U(p)$ in question. An elementary argument shows that this is the whole of $\GL_2(\F_p)$, and so $\omega=0$, which is a contradiction. Therefore, there exists a prime $r\neq p$ such that $a_r(c, \omega)\neq 0$, and the proposition follows.
\end{proof}

Given a point $x\in X_{H}(\Q(\zeta_p)^+)$ and a maximal ideal $\lambda$ of $R$, define
\begin{equation*}
B_{\lambda}(x):=\{y\in X_{H}(\Q): y\equiv x\pmod{\lambda}\}.
\end{equation*}
In other words, $B_{\lambda}(x)$ is the set of $\Q$-rational points in the residue class of $x$ modulo $\lambda$.
\begin{corol}\label{nores}
Let $c$ be a cusp of $X_{H}$ and let $\lambda$ be a maximal ideal of $R$ of characteristic different from $2$ (once again, we note that the characteristic is also not $p$). We then have $B_{\lambda}(c)=\emptyset$.
\end{corol}
\begin{proof}
Say that the characteristic of $\lambda$ is $\ell$. Suppose, for contradiction, that $B_{\lambda}(c)$ is non-empty, and let $x\in B_{\lambda}(c)$. Consider $f_{c}:X_{H/\Q(\zeta_p)}\rightarrow A_{/\Q(\zeta_p)}$. The first observation one must make is that $f_{c}(x)$ is a torsion point in $A(\Q(\zeta_p))$. This is an easy argument that can be found (for the case of a different modular curve) in the paper of Darmon and Merel~\cite{darmer}. It goes as follows. We first show that a multiple of $\iota_c(x)$ is a $\Q$-rational point in $J_{H}$. Indeed, $\iota_c(x)=\mathrm{cl}(x-c)$. Now, given $\sigma\in \Gal(\Q(\zeta_p)/\Q)$, we find
\begin{equation*}
{}^{\sigma}\iota_c(x)-\iota_c(x)=\mathrm{cl}(x-{}^{\sigma} c)-\mathrm{cl}(x-c)=\mathrm{cl}(c-{}^{\sigma}c),
\end{equation*}
because $x$ is defined over $\Q$. The Drinfeld--Manin theorem yields the existence of a positive integer $m_{\sigma}$ such that $m_{\sigma}\cdot \mathrm{cl}(c-{}^{\sigma}c)\in J_{H}(\Q)$. Taking $m:=\max_{\sigma}\{m_{\sigma}\}$, we get
\begin{equation*}
m\cdot{}^{\sigma}\iota_c(x)=m\cdot\iota_c(x)
\end{equation*}
for all $\sigma\in\Gal(\Q(\zeta_p)/\Q)$. It follows that $m\cdot f_c(x)$ is $\Q$-rational. As $A(\Q)$ is finite, $m\cdot f_c(x)$ is torsion, and so is $f_c(x)$ in the first place.

As the characteristic of $\lambda$ is not $2$ nor $p$, reduction modulo $\lambda$ gives rise to an injective group homomorphism
\begin{equation*}
A(\Q(\zeta_p))_{\rm{tors}}\hookrightarrow A_{/k(\lambda)}(k(\lambda)).
\end{equation*}
As $x_{/k(\lambda)}=c_{/k(\lambda)}$ and the image of $c$ in $A$ is $0$, knowing that $f_c(x)$ is a torsion point allows us now to conclude that $f_c(x)=f_c(c)=0$.

To achieve a contradiction, we are now going to make use of Proposition~\ref{formim}. Consider the $R_{\lambda}$ sections of $X_{G(p)}$ defined by $x$ and $c$. Let $h_x:\hat{\loc}_{X_{H},x_{/k(\lambda)}}\rightarrow R_{\lambda}$ be the homomorphism of completed local rings at the special fibres induced by $x$, and let $h_c:\hat{\loc}_{X_{H},c_{/k(\lambda)}}\rightarrow R_{\lambda}$ be that induced by $c$. Note that as $c_{/k(\lambda)}=x_{/\lambda}$, we have $\hat{\loc}_{X_{H},x_{/k(\lambda)}}=\hat{\loc}_{X_{H},x_{/k(\lambda)}}$. The statement that $f_c(x)=f_c(c)$ means that
\begin{equation*}
h_x\circ \hat{f}^{\#}_c= h_c\circ \hat{f}^{\#}_c.
\end{equation*}
But the statement of Proposition~\ref{formim} is precisely that $ \hat{f}^{\#}_c$ is surjective, which leads to $h_x=h_c$. But this is only possible if $x=c$. However, this is a contradiction, as $x$ was assumed to be defined over $\Q$, but the field of definition of $c$ is $\Q(\zeta_p)$. 
\end{proof}

\section{The image of the residual Galois representation in the non-integral case}\label{thmb}
When the prime $p$, besides not being in the set $\{2,3,5,7,13\}$, satisfies $p\equiv 2\pmod{3}$, the group $G(p)$ (as defined in Proposition~\ref{zywina}) is a proper subgroup of $N_{\rm{ns}}(p)$ containing $H(p)$, as the following lemma shows.
\begin{lem}\label{lem1}
If $p\equiv 2\pmod{3}$, the group $H(p)$ is a subgroup of $G(p)$.
\end{lem}
\begin{proof}
Explicitly, the group $H(p)$ is 
\begin{equation*}
\left\{\begin{pmatrix} a & 0 \\ 0 & \pm a\end{pmatrix}:a\in\F_p^{\times}\right\}\cup\left\{\begin{pmatrix} 0 & \epsilon_p b\\ \pm b & 0\end{pmatrix}: b\in\F_p^{\times}\right\}.
\end{equation*}
We only have to show that these elements are contained in $G(p)$. As $p\equiv 2\pmod{3}$, the endomorphism of $\F_p^{\times}$ given by $a\mapsto a^3$ is surjective. It follows from the definition of $G(p)$ that all the elements of the form $\begin{pmatrix} a & 0 \\ 0 & \pm a\end{pmatrix}$, $a\in \F_p^{\times}$, are contained in $G(p)$. Similarly, the function $\F_p^{\times}\rightarrow\F_p^{\times}$ given by $b\mapsto\epsilon_p^2 b^3$ is surjective, from where it follows that all the elements of the type $\begin{pmatrix} 0 & \epsilon_p b\\ \pm b & 0\end{pmatrix}$, $b\in\F_p^{\times}$, are in $G(p)$.
\end{proof}

As a consequence, we can take $H=G(p)$ and the results of the previous sections will hold.

The last ingredient we need to be ready to prove Theorem~\ref{main} (b) is the following result~\cite{lem}. Its usefulness resides in the fact that it implies that, in our situation, our elliptic curve has potentially good reduction at~$p$. For a proof, we refer the reader to~\cite[Proposition~2.2]{lem}.

\begin{prop}\label{notp}
Suppose that $p\geq 5$ and the image of $\bar{\rho}_{E,p}$ is contained in the normaliser of a non-split Cartan subgroup of $\GL(E[p])$. Then $E$ has potentially good reduction at every prime $\ell$ satisfying $\ell\not\equiv \pm 1\pmod{p}$.
\end{prop}

\begin{proof}[Proof of Theorem~\ref{main} (b)]
 Suppose that the image of $\bar{\rho}_{E,p}$ is neither $\GL(E[p])$ nor the normaliser of a non-split Cartan subgroup of $\GL(E[p])$. Zywina's result (Proposition~\ref{zywina}) then asserts that $p\equiv 2\pmod{3}$ and that the image $G$ of $\bar{\rho}_{E,p}$ is conjugate to the group $G(p)$ in $\GL(E[p])$. By choosing an appropriate basis for $\GL(E[p])$, we may in fact assume that the image of $\bar{\rho}_{E,p}$ is contained in $G(p)$.

As we are assuming that $j(E)\notin\Z$, there is some prime $\ell$ such that $v_{\ell}(j(E))<0$. Proposition~\ref{notp} shows that $\ell\equiv\pm 1\pmod{p}$. In particular, $\ell\nmid 2p$. Let $\lambda$ be any prime ideal of $\Z[\zeta_p]$ lying above $\ell$. The elliptic curve $E$ gives rise to a $\Q$-rational point $x$ in $X_{G(p)}$. As $E$ has potentially multiplicative reduction at $\ell$, it follows that there is a cusp $c$ of $X_{G(p)}$ such that $\tilde{x}=\tilde{c}$, where $\tilde{c}$ and $\tilde{x}$ denote the reductions of $c$ and $x$ modulo $\lambda$, respectively. In other words, $x\in B_{\lambda}(c)$. But this contradicts Corollary~\ref{nores}.
\end{proof}

The proof of Theorem~\ref{main} (a)---i.e., the case where $E$ has integral $j$-invariant---will be the subject of the next section.

\section{Runge's method on the curve $X_{G(p)}$ and end of proof of Theorem \ref{main}}
\label{secRunge}

In this section, we deal with the case where $j(E) \in \Z$ and $E$ defines a rational point of the modular curve $X_{G(p)}$, denoted by $P$. As we only need to treat this case, we assume that $p \equiv 2\mod 3$ and $p \notin I(1)$ in all this section. We will prove the following.

\begin{prop}
\label{propRunge}
If $P \in X_{G(p)} (\Q)$ and $j(P) \in \Z$, 
\[
\log |j(P)| \leq 7 \sqrt{p}.
\]
\end{prop}

Before proving this proposition, here are its consequences for the end of the proof of Theorem 1.2.

\begin{corol}
\label{corolisogeny}
For any prime $p \notin I(1)$ congruent to $2 \mod 3$ and any elliptic curve $E$ over $\Q$ without complex multiplication, if the image of $\bar{\rho}_{E,p}$ is included in a conjugate of $G(p)$ and $j(E) \in \Z$, then $p \leq 1.4 \times 10^7$.
\end{corol}

\begin{proof}[Proof of Corollary \ref{corolisogeny}]
Assume that Proposition \ref{propRunge} holds. By the explicit surjectivity theorem of \cite[(7) and Theorem 5.2]{lefourn} (only making use of the fact that the image is contained in the normaliser of a nonsplit Cartan), based on isogeny theorems of Gaudron and Rémond \cite{GaudronRemond14}, we also have (if $\log |j(E)| \geq 12 \cdot 985)$
\[
 p^2 \leq 4 \cdot 10^7 \left( \frac{\log |j(E)|}{12} + 3 + 4 \log(2) \right)^2,
\]
which gives
\[
 \log |j(E)| \geq \frac{6 p}{10^{3.5}} - 70.  
\]
This yields $p \leq 1.4 \cdot 10^7$ when combined with the bound of Proposition \ref{propRunge}. Finally, if $ \log |j(E)| \leq 12 \cdot 985$, we get  by the same surjectivity theorem an absolute upper bound on $p^2$ giving a smaller upper bound on $p$.
\end{proof}

The proof of Proposition \ref{propRunge} relies on Runge's method, which starts with the following fact.

\begin{lem}\label{galcusps}
The set of cusps of $X_{G(p)}$ consists of two Galois orbits. One of these Galois orbits is the set of cusps at infinity, identified via Corollary \ref{corolcusps} with the orbit 
\begin{equation*}
(\loc_{\rm{cubes}})_{/\pm 1} \subset M_p, \quad \loc_{\rm{cubes}} :=\left \{ \begin{pmatrix} a \\ b \end{pmatrix} \in  \F_p^2 \backslash\{(0,0)\} : \,  a+\sqrt{\epsilon_p}b\text{ is a cube in }\F_{p^2}^{\times} \right\}.
\end{equation*}
\end{lem}
\begin{proof}
By Corollary~\ref{galcusps2}, we have a correspondence between the Galois orbits of the set of cusps of $X_{G(p)}$ and the set $G(p)\backslash M_p$. Fix a square root $\sqrt{\epsilon_p}$ of $\epsilon_p$ and consider the one-to-one map
\begin{equation*}
   \theta: M_p\rightarrow \F_{p^2}^{\times}/\{\pm 1\},\quad \begin{pmatrix} a \\b \end{pmatrix}\mapsto a+\sqrt{\epsilon_p}b.
\end{equation*}
Let $\gamma=\begin{pmatrix} x & \sqrt{\epsilon_p}y \\ y & x\end{pmatrix}$ be an element of $C_{\rm{ns}}(p)$. An easy calculation shows that 
\begin{equation*}
    \theta\left(\gamma\begin{pmatrix} a \\ b\end{pmatrix}\right)= (x+\sqrt{\epsilon_p}y)(a+\sqrt{\epsilon_p}b).
\end{equation*}
Moreover, the action of the matrix $\begin{pmatrix}1 & 0 \\ 0 & -1\end{pmatrix}$ corresponds, under $\theta$ to the action of the Frobenius automorphism on $\F_{p^2}^{\times}$. We thus see that $(\loc_{\rm{cubes}})_{/\pm 1}$ is a Galois orbit for the action of $G(p)$ on $M_p$. Indeed, if $\gamma\in G(p)\cap C_{\rm{ns}}(p)$, then the action of $\gamma$ on elements of $M_p$ corresponds, under $\theta$ to multiplication by a cube, and so it is clear that it preserves cubes. Also, the action of Frobenius on $\F_{p^2}^{\times}$ maps cubes to cubes. 

The complement of $(\loc_{\rm{cubes}})_{/ \pm 1}$ in $M_p$ is also a Galois orbit because any non-cube of $\F_{p^2}^{\times}$ can be obtained from a specified non-cube by multiplying by an appropriate cube and applying the Frobenius automorphism if needed. Finally, by definition of cusps at infinity, we see that $(\loc_{\rm{cubes}})_{/ \pm 1}$ is precisely the set of cusps at infinity of $X_{G(p)}$.
\end{proof}

\begin{proof}[Proof of Proposition \ref{propRunge}] 
	
	
	
As we explain right below, Lemma~\ref{galcusps} announces that it is possible to apply Runge's method to integral points of $X_{G(p)}$ with respect to the $j$-invariant, as there are two Galois orbits of cusps. Indeed, one can then construct (following the results of \cite{KubertLang}) a modular unit $U \in \Q(X_{G(p)})$, i.e. a function whose sets of zeros and sets of poles are  the two orbits of cusps. We will then prove that $U \in \overline{\Z[j]}$ and $p^{3}/U \in  \overline{\Z[j]}$ after looking at the $q$-expansions, which directly gives that $U(P)$ is an integer, nonzero and bounded when $j(P) \in \Z$, as $j(P)$ can be large only when $P$ is close from the cusps. As $U$ goes (quickly) to 0 or infinity when $P$ comes close to cusps, studying further the $q$-expansions will then prove that $P$ is far away from the cusps and allow to bound $|j(P)|$. For more general results on Runge's method on modular curves, the reader is invited to consult \cite{BiluParent11} where the applicability domain (and general bound) is given for every modular curve.

	To do it in practice, one needs to define properly a modular unit. We follow the results of \cite{KubertLang} and ideas of Bajolet, Bilu and Matschke \cite{BajoletBiluMatschke2018} here: even though we do not exactly use their own results except at the very end, their arguments definitely inspired the construction of our modular unit. We use the following notations.
	
	$\bullet$ $e(x) := e^{2 i \pi x}$ for any complex number $x$.
	
	$\bullet$ $\Hcal$ is the upper half-plane and $\tau$ will denote any element of $\Hcal$, for which $q_\tau := e(\tau)$ (we will generally drop the subscript when $\tau$ is obvious). For any rational number $r$, the convention is $q_\tau^r := e(r \tau)$.

For any nonzero pair $\underline{a} := (a_1,a_2)$ in $\Q^2 \cap [0,1[^2$ (to simplify notations) with common denominator $p$, we can define a modular function $g_{\underline{a}}$ on $\Hcal$ whose $q$-expansion is
	$\Hcal$ is 
	\begin{equation}
	\label{eqexpg}
	 g_{\underline{a}} (\tau) = q^{B_2(a_1)/2} e(a_2(a_1-1)) \prod_{n=0}^{+ \infty} (1- q^{n+a_1} e(a_2)) (1 - q^{n+1-a_1} e(-a_2)),
	\end{equation}
where $B_2 (X) = X^2 - X + 1/6$ is the second Bernoulli polynomial. There is a modular transformation formula for these units, but we only need the following fact: for $\Ocal$ a subset of $\F_p^2 \backslash \{(0,0)\}$ (stable by $-\Id$ and the action of $G(p)$), choosing for every $(a,b) \in \Ocal$ their lift $(\widetilde{a},\widetilde{b}) \in (\Z \cap [0, p[)^2 $ and then $(a_1,a_2) := \frac{(\widetilde{a},\widetilde{b})}{p}$, and for $m \in \N^*$, the product 
    \begin{equation}
    \label{equnitgen}
     U_{\Ocal,m} := \prod_{(a,b) \in \Ocal} g_{(\tilde{a}/p,\tilde{b}/p)}^m
    \end{equation}
is automorphic of degree 0 for the congruence subgroup associated to $G(p)$ and defines up to multiplication by a root of unity a function on $\Q(X_{G(p)})$ if 
\begin{equation}
\label{eqvanishingsumscond}
 m \sum_{(a,b) \in \Ocal} a^2 = m \sum_{(a,b) \in \Ocal} b^2 = m \sum_{(a,b) \in \Ocal}  a b = 0 \in \Z/p\Z
\end{equation}
and 6 divides $m |\Ocal|$ \cite[Theorem 5.2 in Chapter 3]{KubertLang}.

This is what we use to define our key function (with $\Ocal=\Ocal_{\rm{cubes}}$), whose properties are summed up below.

\begin{prop}
The function on $\Hcal$ defined by 
\[
U := \zeta \cdot U_{\Ocal_{\rm{cubes}},3} = \zeta \cdot \prod_{\substack{(a,b) \in \F_p^2 \\ a + \sqrt{\epsilon_p} b \textrm{ cube in } \F_{p^2}^*}} g_{\frac{(\widetilde{a},\widetilde{b})}{p}}^3
\]
with $\zeta$ the root of unity chosen such that the constant term of the $q$-expansion of $U$ is 1, induces a rational function on $X(p)$ (also denoted by $U$) which satisfies the following properties.
\begin{itemize}
    \item It belongs to $\Q(X_{G(p)})$.
    \item Its zeroes are the cusps at infinity of $X_{G(p)}$, and its poles make up the second Galois orbit of cusps of $X_{G(p)}$.
    \item It is integral over $\Z[j]$ as well as $p^3 U^{-1}$.
\end{itemize}
\end{prop}

\begin{proof}
First, $U$ does define a function on $\Q(X_{G(p)})$ by \cite[Theorem 5.2 in Chapter 3]{KubertLang} . Indeed, $m=3$ is enough, and for the orbit $\Ocal_{\rm{cubes}}$, the vanishing conditions \eqref{eqvanishingsumscond} hold because the set of cubes of $\F_{p^2}^*$ is stable by multiplication by scalars of $\F_{p}^*$, so each of the three sums of \eqref{eqvanishingsumscond} has to be equal to itself times any scalar in $\F_p^*$, hence it is 0.

Regarding the divisors, each of the modular function $g_{\underline{a}}$ is nonvanishing on $\Hcal$ so the divisor of $U$ is supported on the cusps. The analysis of the $q$-expansion at infinity later proves that its image in $X_{G(p)}$ is a zero of $U$, so all its Galois conjugates are as $U$ is $\Q$-rational, and the only other Galois orbit (by Lemma \ref{galcusps}) must be made up with the poles of $U$.

For integrality, each $g_{a_1,a_2}$ is integral over $\Z[j]$ \cite[Theorem 2.2 of Chapter 2]{KubertLang} so $U$ is and it is easily seen from the $q$-expansions that 
\[
\prod_{(a,b) \in M_p} g_{\widetilde{a}/p,\widetilde{b}/p}^3 = \pm p^3,
\]
so $p^3/U$ is also integral over $\Z[j]$.
\end{proof}

Consequently, for every $P \in X_{G(p)}(\Q)$ with $j(P) \in \Z$, 
\begin{equation}
\label{eqUintvalues}
    U(P) \in \Z \textrm{ and } 0 \leq \log |U(P)| \leq 3 \log(p),
\end{equation}
which is the whole point of considering this modular unit.
 We now use the expansion at infinity to bound $q_\tau$. 
    
For every $\tau \in \Hcal$, gathering the $q$-expansions \eqref{eqexpg},
\[
 \log |U(\tau)| = \operatorname{Ord}_q (U) \log |q| + \log |\rho_{U}| + \log |R(\tau)|
\]
with
\[
 \operatorname{Ord}_q (U) = 3 \sum_{(a,b) \in \Ocal_{\textrm{cubes}}} B_2(\widetilde{a}/p)/2, \quad  \rho_U =  \prod_{(a,b) \in {\Ocal_\textrm{cubes}}} \rho_{\widetilde{a}/p,\widetilde{b}/p}^3
\]
and
\[
 \rho_{(a_1,a_2)} = \left\{ \begin{array}{lcr}
                              - e((a_1-1)a_2/2) & \rm{ if } & a_1 \neq 0 \\
                              - 2 i \sin(\pi a_2/2) & \rm{ if } & a_1= 0.
                             \end{array}
\right.
\]
Finally, 
\[
 \log |R(\tau)| = 3 \sum_{(a,b) \in \Ocal_{\textrm{cubes}}} \log |R_{a_1,a_2}(\tau)|
\]

where 
\[
 \log |R_{a_1,a_2}(\tau)| = \sum_{n \geq 0} \log |1 - q^{n+a_1} e(a_2)| +  \log |1 - q^{n+a_1} e(a_2)|.
\]

We will obtain the following estimates and equalities.
\begin{prop}
We have 
\[
\operatorname{Ord}_q(U) = \frac{p^2-1}{4p}, \quad |\rho_U| = (p-1)^3, 
\]
and 
\[
|\log R(\tau) | \leq 2 (p^2-1) \frac{|q|}{1-|q|} + \frac{\pi^2 p (p-2)}{3 |\log|q||}.
\]
\end{prop}

\begin{proof}
First, for $a=0$, all nonzero $b$'s satisfy that $(a,b) \in \Ocal_{\rm{cubes}}$ because $\epsilon_p$ is a cube in $\F_{p^2}^*$ (check its order). This gives $(p-1)$ elements in the orbit. Moreover, $\Ocal_{\rm{cubes}}$ is stable by scalar multiplication by $\F_p^*$, hence all fibers of $(a,b) \mt a$ have the same cardinality except above 0. They are thus of cardinality $(p-2)/3$.  This allows us to compute
\[
 \operatorname{Ord}_q (U) = \frac{3}{12} (p-1) + \frac{(p-2)}{2} \sum_{a=1}^{p-1} ((a/p)^2 - (a/p) + 1/6) = \frac{p^2 - 1}{4p}.
\]
Similarly, for $\rho$, as all terms except for $a_1=0$ have modulus 1,
\[
 |\rho_{U}| = \prod_{b=1}^{p-1} |1 - e(b/p)|^3 = (p-1)^{3}.
\]
Finally, for $R(\tau)$, we use that $|\log |1 - z|| \leq - \log |1 - |z||$ for $|z| \leq 1$ for $n=0$, and $|\log |1 - z|| \leq \frac{|z|}{1 - |z|}$ for the other terms (if $a_1=0$, the first $n=0$ term is put into $\rho_{(0,a_2)}$. We thus get for $a_1 \neq 0$ 
\[
 |\log |R_{a_1,a_2}(\tau)|| \leq |\log(1 - |q|^{a_1})| + |\log(1 - |q|^{1-a_1})|+  \frac{2|q|}{1-|q|},
\]
and 
\[
 |\log |R_{0,a_2}(\tau)|| \leq  \frac{2|q|}{1-|q|}.
\]
Gathering the previous inequalities for the product expansion,
\begin{eqnarray*}
 |\log R(\tau)| & \leq & 2 (p^2-1) \frac{|q|}{1-|q|} + 2 (p-2) \sum_{a=1}^{p-1} |\log (1 - x^a)|, \quad x=|q|^{1/p} \\
 & \leq & 2 (p^2-1) \frac{|q|}{1-|q|} + \frac{\pi^2(p-2)}{3 |\log(x)|} \\
 & \leq & 2 (p^2-1) \frac{|q|}{1-|q|} + \frac{\pi^2 p (p-2)}{3 |\log|q||}
\end{eqnarray*}
by\cite[Lemma 3.5]{bilpar}. 

\end{proof}

Now, assume $\gamma \in \SL_2(\Z)$ is such that its reduction modulo $p$ is of the shape $\begin{pmatrix} a & \epsilon_p b \\ b & a \end{pmatrix}$, where $a + \varepsilon_p b$ is \emph{not} a cube in $\F_{p^2}^*$. The composition $U \circ \gamma$ is a modular unit on $X_{G(p)}$ (not necessarily defined over $\Q$ anymore), but by arguments similar to the previous ones, we have the following: 
\[
\log |U(\gamma \tau)| = \operatorname{Ord}_{{\gamma}} U \cdot \log |q_{\tau}| + \log |\rho_{U,\gamma}| + \log |R_\gamma(\tau)|,
\]
where 
\[
\operatorname{Ord}_{{\gamma}} U = - \frac{p^2 - 1}{8p}, \quad \log |\rho_{U,\gamma}| = 0, \textrm{ and } |\log R_\gamma (\tau)| \leq  \frac{\pi^2 p (p+1)}{3 |\log |q||}.
\]
The argument behind each of those computations is that by our hypothesis on $\gamma$, the function $(a,b) \mapsto a_1 ((a,b) \cdot \gamma))$ on $\Ocal_{\textrm{cubes}}$ does not have 0 in its image, and each other element of $\F_p^*$ has $(p+1)/3$ elements in its fiber (again by stability by multiplication by $\F_p^*$).
%

Putting this together, we obtain 
\[
 \left| \log |U(\tau)| - \frac{p^2-1}{4p} \log |q| - 3 \log(p-1) \right| \leq 2 (p^2-1) \frac{|q|}{1-|q|} + \frac{\pi^2p (p-2)}{3 |\log |q||}
\]
and for the choice of $\gamma$ above, 
\[
 \left| \log |U(\gamma \tau)| + \frac{p^2-1}{8p} \log |q| \right| \leq   \frac{\pi^2 p (p+1)}{3 |\log |q_{\tau}||}.
\]
Now, let us assume that there is a noncuspidal point $P \in X_{G(p)}(\Q)$ with $j(P) \in \Z$. There is a lift $\tau \in \Hcal$ such that $|q_\tau|$ is small and a $\gamma \in \SL_2(\Z)$ such that $\gamma \cdot \tau$ is above $P$ in the complex uniformization of $X_{G(p)}$. This means that $P$ is close to the cusp $\gamma^{-1}(\infty)$. Up to Galois conjugation (which fixes $P$ but changes the cusps), we can reduce to two situations: either $\gamma = \Id$ (which means that $\tau$ belongs to the usual fundamental domain for $\SL_2(\Z)$), or $\gamma$ is chosen as above such that its reduction modulo $p$ corresponds to a matrix of $C_{\rm{ns}}(p)$ not in $G(p)$. In these two cases, we respectively have $U(\tau) = U(P)$ and $U(\gamma \tau) = U(P)$, and this is where we use \eqref{eqUintvalues} to bound the corrresponding term in one of the two previous inequalities. The first case gives
\[
\frac{p^2 -1}{4p} |\log |q|| \leq 3 \log(p-1) + 2 (p^2-1) \frac{|q|}{1-|q|} + \frac{\pi^2p (p-2)}{3 |\log |q||}.
\]
Assuming $p \geq 100$ and $|\log |q|| \geq \sqrt{p}$, we can bound roughly the coefficients and the nondominant terms to obtain  
\[
|\log |q|| \leq 1.2 +  \frac{13 p}{|\log |q||}.
\]
Proceeding similarly in the second case (with the same assumptions on $p$ and $|q|$), we obtain 
\[
 |\log |q|| \leq 1.2 + \frac{27 p}{|\log |q||}.
\]
Both cases give rise to second-degree polynomial inequalities which we can readily solve, and using then the estimates of \cite[Corollary 2.2]{BiluParent11}, after simplification, 
\[
 \log |j(P)| \leq 7 \sqrt{p}.
\]
We can retrieve the remaining cases $p < 100$ by refining the estimates above (or by using the main theorem of \cite{BajoletBiluMatschke2018}), and the case $\log |q| \leq \sqrt{p}$ by \cite[Corollary 2.2]{BiluParent11} again, which concludes the proof.
\end{proof}

\appendix
\section{The proof of Lemma~\ref{imgeta}}\label{appa}

\begin{proof}[Proof of Lemma~\ref{imgeta}]
  Using the identification of Corollary~\ref{corolcusps}, the action of the map $d_1$ on the cusps corresponds to the canonical projection
 \[ C_{\rm{sp}}(p)\backslash M_p\times\F_p^{\times}\rightarrow T(p)\backslash M_p\times\F_p^{\times},\] where $T(p)$ is the upper triangular subgroup of $\GL_2(\F_p)$. Using equation (\ref{idents}), we see that $d_p=w_p\circ d_1\circ \omega_p$. One easily checks that the action of $\omega_p$ on the cusps of $X_{\rm{sp}}(p)$ is given by
 \[ \omega_p:\left(\begin{pmatrix} a \\ b\end{pmatrix},d\right)\mapsto\left(\begin{pmatrix} b \\ a\end{pmatrix},-d\right).\]
 
 Let $c$ be the cusp at infinity of $X_{\rm{sp}}^+(p)$, and let $c'$ be the cusp at infinity of $X_{\rm{sp}}(p)$ (which obviously lies over $c$). We will show that $\eta'(c')=0$, from where it will immediately follow that $\eta(c)=0$.
 
 The image of $c'$ under $d_1$ is the cusp at infinity of $X_0(p)$. To compute $d_p(c')$, we start by calculating $\omega_p(c')$. Being a cusp at infinity, $c'$ is represented by an element of $M_p\times\F_p^{\times}$ of the form
 \[ \left(\begin{pmatrix} a \\ 0 \end{pmatrix},a\right)\quad a\in \F_p^{\times}.\]The involution $\omega_p$ then maps $c'$ to the cusp represented by $\left(\begin{pmatrix} 0\\ a\end{pmatrix},-a\right)$. The image of this cusp under $d_1$ is the cusp $0$ of $X_0(p)$. As $w_p$ swaps the cusp $0$ with $\infty$, we conclude that $d_p(c')$ is also the cusp at infinity of $X_0(p)$. It follows from the definition of $\eta'$ that $\eta'(c')=0$, as we wanted.
 
 Now let $c$ be a cusp of $X_{\rm{sp}}^+(p)$ not at infinity. Let $c'$ be a cusp of $X_{\rm{sp}}(p)$ lying over it (in this situation, $c'$ is necessarily not at infinity). We will now show that $\eta'(c')=\mathrm{cl}(0-\infty)$, which will conclude the proof of the lemma.
 
 Let
 \[ \left(\begin{pmatrix} a \\ b\end{pmatrix}, d\right),\quad a, b,d\in\F_p\] be an element of $M_p\times\F_p^{\times}$ representing the cusp $c'$. Note that $a,b\in\F_p^{\times}$ because otherwise $c'$ would lie over the cusp at infinity of $X_{\rm{sp}}^+(p)$.  As $b\in \F_p^{\times}$, the image under $d_1$ of $c'$ is the cusp $0$. The cusp $\omega_p(c')$ is represented by $\left(\begin{pmatrix} b \\ a \end{pmatrix},-d\right)$. As $a\in\F_p^{\times}$, the image of $\omega_p(c')$  under $d_1$ is the cusp $0$. Using again the fact that $w_p$ swaps the cusps of $X_0(p)$, we conclude that $d_p(c')=\infty$. Therefore, $\eta'(c')=\mathrm{cl}(0-\infty)$, as we wanted.
\end{proof}
\section{The proof of Proposition~\ref{zywina}}\label{appb}

Following a suggestion made by an anonymous referee, we add here, as an appendix, Zywina's proof of Proposition~\ref{zywina} (which is Proposition~1.13 in his paper~\cite{zyw}). We must emphasise that none of the results and ideas in this appendix are due to the authors of this paper, and that the original version of this proof can be found in~\cite{zyw}. The main reason for the existence of this appendix is the fact that Zywina's paper remains unpublished.

 Fix a decomposition subgroup $D_p$ of $G_{\Q}$ over $p$, and let $I$ be the corresponding inertia subgroup.
\begin{proof}[Proof of Proposition~\ref{zywina}]
 As $p$ does not lie in $I(1)$, we know (Theorem~\ref{soa}) that the image of $\bar{\rho}_{E,p}$ is contained in the normaliser of a non-split Cartan subgroup. By choosing a basis of $E[p]$ appropriately, we may assume that this normaliser of non-split Cartan is $N_{\rm{ns}}(p)$. Let $j(E)$ be the $j$-invariant of $E$. We start by showing that we must have $v_p(j(E))\geq 0$.
\newline

Suppose that $v_p(j(E))<0$. Our elliptic curve $E_{/\Q_p}$ is either isomorphic to a Tate curve over $\Q_p$, or is a quadratic twist of one. Therefore, there exists a character $\psi:D_p\rightarrow \F_p^{\times}$, trivial or quadratic, such that 
\[ \bar{\rho}_{E,p}|_{D_p}\sim \begin{pmatrix} \psi\chi_p & * \\ 0 & \psi\end{pmatrix}.\] As $\chi_p:D_p\rightarrow \F_p^{\times}$ is surjective, it follows that the image of $\bar{\rho}_{E,p}(D_p)$ in $\PGL_2(\F_p)$ has order divisible by $p-1$. Note however that the image of $N_{\rm{ns}}(p)$ in $\PGL_2(\F_p)$ has order $2(p+1)$. If the image of $\bar{\rho}_{E,p}$ were contained in $N_{\rm{ns}}(p)$, then $p-1$ would divide $2(p+1)$, which is not possible because $p\geq 19$. This leads us to conclude that if the image of $\bar{\rho}_{E,p}$ is contained in the normaliser of a non-split Cartan subgroup, then $v_p(j(E))\geq 0$, as we wanted.
\newline

Now that we know that $v_p(j(E))\geq 0$, we will show that $\bar{\rho}_{E,p}(I)$ is a subgroup of index~$1$ or $3$ of $C_{\rm{ns}}(p)$. Before proving this, we point out that $\bar{\rho}_{E,p}(I)$ is cyclic. Indeed, the representation $\bar{\rho}_{E,p}|_I$ factors through the tame inertia subgroup of $I$ because the order of $N_{\rm{ns}}(p)$ is not divisible by $p$. The cyclicity of $\bar{\rho}_{E,p}(I)$ now follows from the fact that the tame inertia subgroup is pro-cyclic. 

There is a finite extension $K$ of $\Q_p$ of ramification degree $e\in\{1,2,3,4,6\}$ over which~$E$ acquires good reduction (see section~5.6 of \cite{ser_prop}). We will denote by $v$ the valuation of $K$ normalised so that $v(p)=e$. Let $I_K$ denote the inertia subgroup of $\Gal(\overline{\Q}_p/K)$. Let \[[p](X)=\sum_{i=1}^{\infty} a_i X^i,\quad a_i\in\Z_p\] be the multiplication by $p$  in the formal group of $E$. As every $a_i$ lies in $\Z_p$, the integers $v(a_i)$ are non-negative multiples of $e$. As either $v(a_p)=0$ (the ordinary case), or $v(a_p)\neq 0$ and $v(a_{p^2})=0$ (the supersingular case), the Newton polygon of $[p](X)$ can then either start with a line segment connecting $(1,e)$ to $(p,0)$, or with a line segment connecting $(1,e)$ to $(p^2,0)$. Using \cite[Proposition 10, section 1.10]{ser_prop} and the fact that the representation
\[\bar{\rho}_{E,p}|_{I_K}:I_K\rightarrow \GL_2(\F_p)\] is semisimple (because the order of $N_{\rm{ns}}(p)$ is coprime to $p$), we conclude that in the ordinary case we have \cite[Proposition 11, section 1.11]{ser_prop}
\begin{equation}\label{firstcase}\bar{\rho}_{E,p}|_{I_K}\sim \begin{pmatrix} \chi_p & 0 \\ 0 & 1\end{pmatrix},\end{equation} while in the supersingular case we have
\begin{equation}\label{secondcase}\bar{\rho}_{E,p}|_{I_K}\otimes_{\F_p}\overline{\F}_p\sim \begin{pmatrix} \theta_2^e & 0 \\ 0 & \theta_2^{pe}\end{pmatrix},\end{equation}where $\theta_2$ is a fundamental character of level $2$. We show that (\ref{firstcase}) cannot occur.

If $\bar{\rho}_{E,p}|_{I_K}$ were as in (\ref{firstcase}), then the image of $\bar{\rho}_{E,p}(I_K)$ in $\PGL_2(\F_p)$ would be a cyclic group of order $p-1$. Since the square of any element in $N_{\rm{ns}}(p)-C_{\rm{ns}}(p)$ is a scalar matrix, the order of every element in the image of $N_{\rm{ns}}(p)$ in $\PGL_2(\F_p)$ divides $p+1$. In particular, we would have $p-1\mid p+1$. However, this is not possible, as $p\geq 19$.

We are thus in situation (\ref{secondcase}). The group $\bar{\rho}_{E,p}(I_K)$ is therefore a cyclic group of order 
\[\frac{p^2-1}{\gcd(p^2-1,e)},\] and so the order of $\bar{\rho}_{E,p}(I)$ is a multiple of this number. Also, it follows that $\bar{\rho}_{E,p}(I)$ is contained in $C_{\rm{ns}}(p)$. Indeed, a generator of $\bar{\rho}_{E,p}(I)$ must have order $(p^2-1)/\gcd(p^2-1, e)\geq (p^2-1)/6$, but every element of $N_{\rm{ns}}(p)-C_{\rm{ns}}(p)$ has order dividing $2(p-1)$. As $p\geq 19$, it follows that a generator of $\bar{\rho}_{E,p}(I)$ must be an element of $C_{\rm{ns}}(p)$, and so $\bar{\rho}_{E,p}(I)$ is contained in $C_{\rm{ns}}(p)$.

As the order of $C_{\rm{ns}}(p)$ is $p^2-1$, we can therefore conclude that $\bar{\rho}_{E,p}(I)$ is a subgroup of index $1,2,3,4$ or $6$ of $C_{\rm{ns}}(p)$. Note however that if this index were even, then $\bar{\rho}_{E,p}(I)$ would be contained in the subgroup of squares of $C_{\rm{ns}}(p)$ (as $C_{\rm{ns}}(p)$ is cyclic, this is the only subgroup of index $2$). However, this would contradict the fact that the determinant of $\bar{\rho}_{E,p}|_{I}$ surjects to $\F_p^{\times}$. Therefore, the index of $\bar{\rho}_{E,p}(I)$ in $C_{\rm{ns}}(p)$ is odd, and so it is $1$ or~$3$.

Now let $H$ denote the group $\bar{\rho}_{E,p}(G_{\Q})\cap C_{\rm{ns}}(p)$. By what we have seen, $H$ has index $1$ or $3$ in $C_{\rm{ns}}(p)$ (it contains $\bar{\rho}_{E,p}(I)$). Note that if $H=C_{\rm{ns}}(p)$, then $\bar{\rho}_{E,p}(G_{\Q})=N_{\rm{ns}}(p)$, as the image of $\bar{\rho}_{E,p}$ cannot be contained in $C_{\rm{ns}}(p)$ (recall that this is due to the fact that the image of complex conjugation must have trace $0$ and determinant $-1$, and there is no element in $C_{\rm{ns}}(p)$ simultaneously satisfying both of these properties). It remains to treat the case where $H$ has index $3$ in $C_{\rm{ns}}(p)$.

Suppose that $H$ has index $3$ in $C_{\rm{ns}}(p)$. As $C_{\rm{ns}}(p)$ is cyclic, there is only one subgroup of index $3$: the subgroup of cubes. As $H$ contains $\bar{\rho}_{E,p}(I)$, the $\det(H)=\F_p^{\times}$. Note that if $p\equiv 1\pmod{3}$, then it is not possible to have $H$ of index $3$ in $C_{\rm{ns}}(p)$, because then $\det(H)\neq \F_p^{\times}$. Thus, if $p\equiv 1\pmod{3}$, the image of $\bar{\rho}_{E,p}$ must be $N_{\rm{ns}}(p)$.

Suppose then that $p\equiv 2\pmod{3}$ and that $H$ has index $3$ in $C_{\rm{ns}}(p)$. It is easy to check that $N_{\rm{ns}}(p)/H$ is isomorphic to $D_3$, the dihedral group of size $6$. The image of $\bar{\rho}_{E,p}(G_{\Q})$ in $N_{\rm{ns}}(p)/H$ has index $3$. Now, all the index $3$ subgroups of $D_3$ are conjugate, from where it follows that the image of $\bar{\rho}_{E,p}$ is a conjugate of $G(p)$ in $\GL_2(\F_p)$.
\end{proof}

\section*{Acknowledgements}
This project started when the second named author was still a postdoctoral fellow at the Max Planck Institute for Mathematics (MPIM). For this reason, he would like to thank the MPIM for their financial support and for the excellent work environment provided. The authors are also grateful to the anonymous referee, who provided thorough feedback on a previous version of this paper and pointed out a generalisation of an earlier result (this is now Proposition~\ref{MW}), helping to improve the quality of the paper.

\bibliography{image}
\bibliographystyle{abbrv}

\end{document}